\documentclass[12pt]{article}
\usepackage{amssymb, amsfonts, amsmath}
\usepackage{hyperref}
\usepackage{xcolor}
\hypersetup{
	colorlinks,
	linkcolor={red!60!black},
	citecolor={green!60!black}
}
\newcommand{\R}{{\mathbb R}}

\newcommand{\N}{{\mathbb N}}

\newcommand{\E}{{\mathbb E}}

\newcommand{\PP}{{\mathbb P}}

\newcommand{\func}[2]{#1 \left( #2 \right)}

\newcommand{\parent}[1]{ \left( #1 \right)}

\newcommand{\setbuilder}[2]{ \left\{ #1 \mid #2 \right\}}

\numberwithin{equation}{section}

\begin{document}
\newtheorem{theorem}{Theorem}[section]
\newtheorem{definition}[theorem]{Definition}
\newtheorem{lemma}[theorem]{Lemma}
\newtheorem{note}[theorem]{Note}
\newtheorem{corollary}[theorem]{Corollary}
\newtheorem{proposition}[theorem]{Proposition}
\renewcommand{\theequation}{\arabic{section}.\arabic{equation}}
\newcommand{\newsection}[1]{\setcounter{equation}{0} \section{#1}}
%%%%%%%%%%%% title %%%%%%%%%%%%%%%%%%%%%%%%%%%%%%%%
\title{The Kac formula and Poincar\'{e} recurrence theorem in Riesz spaces\footnote{This research was funded in part by the joint South Africa - Tunisia Grant (SA-NRF: SATN180717350298, Grant number 120112.). \\ {\bf Keywords:} Kac's formula, Riesz spaces, vector lattices, ergodic theory, recurrence.\\ {\bf MSC(2020): } 47B60, 37A30, 47A35, 60A10.}}
%%%%%%%%%%%%%%%%%%%%%%%%%%%%%%%%%%%%%%%%%%%%%%%%
\author{
 Youssef Azouzi ${}^\dagger$\\
 Mohamed Amine Ben Amor ${}^\dagger$\\ 
 Jonathan Homann  ${}^\sharp$ \\ 
 Marwa Masmoudi ${}^\dagger$ \\
  Bruce A. Watson  ${}^\sharp$\\ \\
 ${}^\sharp$ School of Mathematics\\
 University of the Witwatersrand\\
 Private Bag 3, P O WITS 2050, South Africa\\ \\
 ${}^\dagger$ LATAO Laboratory,\\
Faculty of Mathematical, Physical and Natural Sciences of Tunis\\
Tunis-El Manar University,\\ 2092 El Manar, Tunisia\\ }

\maketitle

\begin{abstract}\noindent
Riesz space (non-pointwise) generalizations for iterative processes are given for the concepts of recurrence, first recurrence and conditional ergodicity. Riesz space conditional versions of the Poincar\'{e} Recurrence Theorem and the Kac formula are developed.
Under mild assumptions, it is shown that every conditional expectation preserving process is conditionally ergodic with respect to the conditional expectation generated by the Ces\`aro mean associated with the iterates of the process. 
Applied to processes in $L^1(\Omega,{\mathcal A},\mu)$, where $\mu$ is a probability measure,  new conditional versions of the above theorems are obtained.
\end{abstract}

\parindent=0cm
\parskip=0.5cm

%%%%%%%%%%%%%%%%%%%%%%%%%%%%%%%%%%%%%%%%%%

%%%%%%%%%%%%%%%%%%%%%%%%%%%%%

\section{Introduction} \label{s: introduction}

Riesz space generalizations of stochastic processes have been well studied, see for example \cite{klw-1, stoica, troitski} for some of the earliest and recently \cite{grobler}. 
In this setting analogues of the Hopf-Garsia, Birkhoff and Wiener-Kakutani-Yoshida ergodic theorems were given. More recently mixing processes were considered, see \cite{krw} and \cite{hkw-1}, which revisiting the concept of ergodicity in Riesz space, see \cite{hkw-2}. 
The current work builds on this foundation to consider in the Riesz space setting a  Poincare Recurrence Theorem and the Kac formula for the (conditional) mean of the recurrence time of a conditionally ergodic process. 
 We refer the reader to \cite[pages 67-103]{EFHN} and \cite[pages 33-48]{petersen}, non-conditional for the measure theoretic versions of these results.
Fundamental to this is the Riesz space analogue of the $L^p$ spaces (in particular $L^1$), introduced in \cite{expectation paper} and further studied in \cite{AT} and \cite{krw}. The measure theoretic version of this spatial extension with respect to a conditional expectation operator was considered in \cite{dPG, HdP, dPDH}. 
It is also shown, under mild assumptions, that every conditional expectation preserving process is conditionally ergodic with respect to the conditional expectation generated by the Ces\`aro mean associated with the iterates of the process. 
When applied to processes in $L^1(\Omega,{\mathcal A},\mu)$, where $\mu$ is a probability measure,  new conditional versions of the above theorems are obtained.  In particular, the Riesz space version of the Kac formula  yields a conditional Kac formula measure preserving systems.

The starting point of working with stochastic processes in a Riesz space is the definition of conditional expectation operators on Riesz spaces.  We use the definition of \cite{klw-1},  which when restricted to a probability space with the prefer weak order unit the a.e. constant $1$ function, yields the classical definition of a conditional expectation.

\begin{definition}
Let $E$ be a Riesz space with weak order unit. A positive order continuous projection $T \colon E \rightarrow E$, with range, $R(T)$, a Dedekind complete Riesz subspace of $E$, is called a conditional expectation if $Te$ is a weak order unit of $E$ for each weak order unit $e$ of $E$.
\end{definition}

We recall, see \cite{petersen}, that  $\parent{\Omega,\3B,\mu,\tau}$ is called a measure preserving system if $\parent{\Omega,\3B,\mu}$ is a probability space and ${\tau}:\Omega \rightarrow \Omega$ is a mapping with $\func{\mu}{\func{{\tau}^{-1}}{A}}=\func{\mu}{A}$ for each $A \in \3B$.
A Riesz space generalization of the concept of a measure preserving system was introduced in \cite{hkw-1} as a conditional expectation preserving system, see below.

\begin{definition} \label{system definition}
The 4-tuple, $\parent{E,T,S,e}$, is called a conditional expectation preserving system if $E$ is a Dedekind complete Riesz space, $e$ is a weak order unit for $E$, $T$ is a conditional expectation operator in $E$ with $Te=e$ and $S$ is an order continuous Riesz homomorphism on $E$ with $Se=e$ and $TS=T$.
\end{definition}

By Freudenthal's theorem the condition $TSf=Tf$ for all $f \in E$ in the above definition is equivalent to $TSPe=TPe$ for all band projections $P$ on $E$.

Let  $(E,T,S,e)$ be a conditional expectation preserving system.
For $f \in E$ and $n \in \mathbb{N}$ we denote
	\begin{equation}% \label{eqn def}
	S_nf := \frac{1}{n}\sum_{k=0}^{n-1} S^kf.
	\label{s_n eqn def}
	\end{equation}
	We set
	\begin{equation}
	L_Sf := \lim_{n \rightarrow \infty} S_n f,	
	\label{fhat eqn def}
	\end{equation}
where the above limit is the order limit and $f\in \3E_S$ where $\3E_S$ is the set of $f\in E$ for which the above order limit exists.  Thus $L_S:\3E_S\to E$.
The set of $S$-invariant $f \in E$ will be denoted 
$\3I_S := \setbuilder{f \in E}{Sf=f}$.
By  \cite[Lemma 2.3]{hkw-2}, $\3I_S\subset \3E_S$ and  $L_Sf=f$ for all $f\in {\mathcal I}_S$. Further, it is easily seen that if $f\in \3E_S$ and $L_Sf=f$ then $f\in {\mathcal I}_S$, see \cite[Theorem 2.4]{hkw-2}. These ideas are further expanded in Section \ref{sec-pre}.

In a measure preserving system $\parent{\Omega,\3B,\mu,\tau}$, a set $B\in \3B$ is said to be $\tau$-invariant if $\mu(\tau^{-1}(B)\Delta B)=0$, i.e. $\chi_B=\chi_{\tau^{-1}(B)}$ a.e.. If we take $Sf:=f\circ \tau$ then $\chi_{\tau^{-1}(B)}=S\chi_{B}$ and, in the $L^1(\Omega,\3B,\mu)$ sense,  $\tau$ invariance of $B$ is equivalent to $S$ invariance of $\chi_B$. 
The measure preserving system $\parent{\Omega,\3B,\mu,\tau}$ said to be ergodic, see \cite[page 42]{petersen}, if every $\tau$-invariant set has measure $0$ or $1$.  This is equivalent to $\chi_B$ being the zero or $1$ constant function in the $L^1(\Omega,\3B,\mu)$ sense. This can be equivalently stated as $\chi_B$ being a constant function in the $L^1(\Omega,\3B,\mu)$ sense, since the $0$ and the $1$ functions are the only available possibilities in the a.e. sense.  Taking $T$ as the expectation operator on $L^1(\Omega,\3B,\mu)$, with range the a.e. constant functions, we have that  $\parent{\Omega,\3B,\mu,\tau}$ is ergodic if and only if $\chi_B\in R(T)$ for each $S$ invariant $\chi_B$.
This equates in the Riesz space sense to each $S$ invariant component of the weak order unit $e$ being in $R(T)$. However, as $S$ is a Riesz homomorphism, the $S$ invariant elements form a Riesz subspace generated %(explain how) 
by the $S$ invariant components of $e$. 
Applying Freudenthal's Theorem and using that $e$ is a weak order unit, we have that 
$\3I_S\subset \func{R}{T}$ if and only of each $S$ invariant component of $e$ is in $R(T)$. Thus we have the following definition of conditional ergodicity in a Riesz space, which is equivalent to that given in \cite{hkw-2}, since $L_Sf=f$ if and only if $f\in\3I_S$.

\begin{definition}[Ergodicity] \label{ergodicity def}
The conditional expectation preserving system $\parent{E,T,S,e}$ is said to be $T$-conditionally ergodic if $\3I_S\subset \func{R}{T}$.
\end{definition}

The notion of recurrence appears in ergodic theory,  but only a measure preserving system is required to define recurrence. 
In particular, see \cite[Definition 3.1, pg 34]{petersen}, for
$\parent{\Omega,\3B,\mu,\tau}$ a measure preserving system and $B \in \3B$, a point $x \in B$ is said to be recurrent with respect to $B$ if there is a $k \in \2N$ for which ${\tau}^{k}x \in B$.
In this setting, see \cite[Theorem 3.2, pg 34]{petersen},  
the Poincar\'{e} Recurrence Theorem gives that
for each $B \in \3B$ almost every point of $B$ is recurrent with respect to $B$.

The concept of recurrence in a Riesz space will be given in terms of components of a chosen weak order unit.
A component $p$ of $q$ which is a component of $e$ will be said to be recurrent with respect to $q$ if $p$ can be decomposed into a countable sequence of components, each one of which maps to a component of $q$ under some iterate of the map $S$. This is formally defined below.

\begin{definition}[Recurrence] \label{recurrence definition}
Let $\parent{E,T,S,e}$ be a conditional expectation preserving system and $p,q$ be components of $e$ with $p\le q$.
We say that $p$ is recurrent with respect to $q$ if there are components $p_n, n\in\N,$ of $p$ so that $\displaystyle{\bigvee_{n\in\N} p_n=p}$ and  $S^np_n\le q$ for each $n\in\mathbb{N}$.
\end{definition}

To understand the connection with the classical measure theoretic definitions we consider the simplest case, that of  $(\Omega,\Sigma,\mu)$ a probability space
with a set $B\in \Sigma$ having $\mu(B)>0$. 
We take as the Riesz space $E=L^1(\Omega,\Sigma,\mu)$,  $e$ as the function on $\Omega$ which takes the value $1$ a.e.  on $\Omega$. 
If $\tau$ is a measure preserving transformation on $\Omega$ we take $Sf(x)=f(\tau(x))$ for each $x\in\Omega$ and $f\in E$.  
For $A\in \Sigma$ with $B\subset A$ let $p=\chi_B$ and $q=\chi_A$ then we have that $B$ is recurrent with respect to $A$ if and only if $B$ can be decomposed into a countable union of measurable sets $B_n$ so that 
$\tau^n(B_n)\subset A$, which is equivalent to
$$B\subset \bigcup_{n=1}^\infty \tau^{-n}(A).$$
Here the Poincar\'{e} Recurrence Theorem gives that if $x \in A \in \3B$ then ${\tau}^kx \in A$ for some $k \in \N$,
and $k$ is a time taken to recur. These ideas are taken to the Riesz space setting in Section \ref{sec-poin} where a Riesz space version of the Poincar\`{e} Recurrence Theorem is presented in Theorem \ref{thm-poincare}.

If $x$ is recurrent with respect to $A$ there will be multiple recurrence times, but what is of interest to us is the first recurrence time.
The time of first entry of $x$ into $A$ is given by
$$n_A(x):=\inf\{n\in\mathbb{N}\,|\,\tau^nx\in A\}$$
and $x\notin A$ we set $n_A(x)=0$. 
The function $n_A$ is called the first recurrence time with respect to $A$.

If we take
$$N_A(n)=\bigcap_{j=1}^n (A\setminus\tau^{-j}(A))= \bigcup\{ B\in \3B\,|\,B\subset A,\, A\cap \tau^j(B)=\emptyset\,\forall j=1,\dots,n\}$$
then $N_A(n)$ is the set of points of $A$ not recurrent in the first $n$ iterates.
Thus for $x\in \Omega$, the first recurrence time of $x$ with respect to $A$ is
$$n_A(x)=\sum_{x\in N_A(k-1)\setminus N_A(k)\atop k\in\N} k$$
where we have set $N_A(0)=A$.
The Kac Formula, \cite[Theorem 4.6, pg 46]{petersen},  gives that the spatial average over $\Omega$ of the first recurrence time with respect to $A$, with $\mu(A)>0$, is $1$, i.e. ${\int_\Omega n_A\,d\mu=1}.$

We now extend these concepts to the Riesz space setting.
For $S$ a bijective Riesz homomorphism and $p$ a component of $e$ in $E$, we set
\begin{equation}\label{np}
N_p(n)=\bigvee\{q\,|\, q \mbox{ a component of } p \mbox{ in } E, p\wedge S^jq =0 \mbox{ for all } j=1,\dots n\}
\end{equation}
and 
\begin{equation}
N_p(0)=p.\label{np0}
\end{equation}
Here $N_p(n)$ is the maximal component of $p$ to have no component recurrent with respect to $p$ in under (less than) $n+1$ iterates of $S$.
Further $N_p(n-1)-N_p(n)$ is the maximal component of $p$ to be recurrent with respect to $p$ in exactly $n$ iterates of $S$. We thus define first recurrence time as 
$$n_p=\sum_{k=1}^\infty k (N_p(k-1)-N_p(k))$$
Here $n_p$ exists in $E_u^+$, the positive cone of the universal completion of $E$, as $N_p(k-1)-N_p(k), k\in\N$, are disjoint components of $e$ in $E$, see \cite[page 363]{L-Z}.
Having given a form for $n_p$ which exists in $E_u^+$ we can conveniently rewrite $n_p$ as
\begin{equation}\label{n-p}
n_p=\sum_{k=0}^\infty N_p(k).
\end{equation}
These concepts form the focus of Section \ref{sec-kac} which culminates in
Theorem \ref{kac-2}, a conditional Riesz space version of Kac's Lemma which does not require the concept of ergodicity. For conditionally ergodic processes Theorem \ref{kac-2} yields Corollary \ref{kac-3}, which is a conditional Riesz space analogue of the classical Kac Lemma.

We complete the work with an application of these results to probabilitistic systems in Section \ref{sec-appl}. Of particular interest here is the extension of Kac's Lemma to conditionally ergodic processes.

%%%%%%%%%%%%%%%%%%%%%%%%%%%%%%%%%%%%%%%%%%%%%%%
\section{Preliminaries} \label{sec-pre}
%----------------------------------------------------------------------------------------------------------%
Background material on Riesz spaces can be found in \cite{riesz book}.
We say that a positive operator $U$ on the Riesz space, $E$, is strictly positive if $Uf=0$ for $f\in E^+$ implies  $f=0$. We denote by $C_p(E)$ the set of components of $p$ in $E$.  
We note that if $e$ is a weak order unit for $E$, then $E_e$, the ideal in $E$ generated by $e$, has an $f$-algebra structure in which $e$ is the algebraic unit and for $p,q\in C_e(E)$ we have that $p\wedge q =pq$.

\begin{lemma}\label{injective}
Let $\parent{E,T,S,e}$ be a conditional expectation preserving system  with $T$ strictly positive,  then $S$ is injective.
\end{lemma}

\begin{proof}
As $S$ is a Riesz homomorphism $S(|f|)=|Sf|$. Thus if 
$f\in E$ with $Sf=0$,  then  
$S(|f|)=|Sf|=0.$
Hence $T|f|=TS|f|=0$,  but $T$ is strictly positive so $|f|=0$ giving $f=0$.
\end{proof}

For $\parent{E,T,S,e}$ a conditional expectation preserving system, if $S$ is bijective then the map $S^{-1}$ is a bijective lattice homomorphism which has $S^{-1}e=e$ and $TS^{-1}=T$ making $\parent{E,T,S^{-1},e}$ a conditional expectation preserving system.
Further, if $Sa\wedge Sb=0$ then $a\wedge b=0$ and  $Su$ is a component of $Sx$ if and only if $u$ is a component of $x$.
In particular,  $S$ and $S^{-1}$ map components of $e$ to components of $e$.

\begin{lemma} \label{lemma1}
Let $S$ be a bijective Riesz homomorphism, $p$ be a component of $e$ and let $a$ be a component of $p.$
 Then $a$ has a unique decomposition $a=b+c$ in two disjoint components of $a$ (and hence also of $p$ and of $e$) with
$p\wedge Sb=Sb$ and $p\wedge Sc=0.$ 
\end{lemma}

\begin{proof}
Let $b=a\wedge S^{-1}p$ and $c=a\wedge S^{-1}(e-p)$.  As $S^{-1}e=e$, and $p, e-p$ are components of $e$ we have that $S^{-1}p$ and $S^{-1}(e-p)$ are components of $e$. Thus $b$ and $c$ are components of $a$ 
with $b\wedge c = a\wedge S^{-1}(p\wedge (e-p))=0$.
Here $b+c=a\wedge S^{-1}(p+(e-p))=a\wedge S^{-1}e=a$ with
$Sc=Sa\wedge (e-p)$ making $Sc\le e-p$ so $p\wedge Sc=0$
and $S^{-1}(p\wedge Sb)=S^{-1}p\wedge a\wedge S^{-1}p=a\wedge S^{-1}p=b$
giving $p\wedge Sb=Sb$.

As for uniqueness, if $a=b'+c'$ with $b'\wedge c'=0$, where $b'$ and $c'$ are components of $a$, with
 $Sb\wedge Sc=0$, $Sb'\wedge Sc'=0$, $p\wedge Sb'=Sb'$ and $p\wedge Sc'=0$ then 
$Sb+Sc=Sb'+Sc'$ and as 
$$Sb=p\wedge(Sb+Sc)=p\wedge(Sb'+Sc')=Sb'$$
making $b=b'$ and thus $c=c'$. 
\end{proof}

\begin{note}\label{note-comp}
	If $Sa \le a$ or $a\le Sa$ for some component $a$, then $a=Sa$.
	To see this, observe that $T(a-Sa)=Ta-TSa=0$ but as $a-Sa\ge 0$ or $Sa-a\ge 0$ from the  strict positivity of $T$ this gives $Sa-a=0$.
\end{note}

The following result, while implicitly contained in various of our previous papers is explicitly needed in the following form in the current work.

\begin{lemma}\label{averaging}
 If $T$ is a strictly positive conditional expectation operator on a Dedekind complete Riesz space with weak order unit $e=Te$, then for each $g\in E_+$ we have that $P_{Tg}\ge P_g$ where $P_{Tg}$ and $P_g$ denote the band projections onto the bands generated by $Tg$ and $g$, respectively.
 \end{lemma}
 
 \begin{proof}
 We recall here,  see \cite[Lemma 3.1 (b)]{expectation paper}, that if $f\in R(T)_+$ then $P_fe\in R(T)$.
  The averaging property of the conditional expectation operator, $T$,  in Riesz spaces,  see \cite[Theorem 4.3]{expectation paper},  gives that if $f, g,fg\in E$ with $f\in R(T)$ then $T(fg)=f T(g)$. 
  Here the multiplicative structure in $E$ is the order continuous linear extension of the multiplication which sets $Q_1e\cdot Q_2e:=Q_1Q_2e$ for band projections $Q_1$ and $Q_2$ on $E$. 
  In particular this makes $E$ into a module over the ring $E_e$,  the ideal generated by $e$ in $E$,  and $P_fg=P_fe\cdot g$ for all $g\in E$ and $f\in E_+$.  In addition $E_e$ is an $f$-algebra with respect to this multiplication. 
  
  For $g\in E_+$,  we have $g\cdot (I-P_{Tg})e=(I-P_{Tg})g\in E$.  Since $Tg\in R(T)$, we have $P_{Tg}e\in R(T)$ giving $(I-P_{Tg})e\in R(T)$ so, by the averaging property, 
  $$T(g\cdot (I-P_{Tg})e)=(I-P_{Tg})e\cdot Tg=(I-P_{Tg})Tg =0.$$
  Thus $T(g\cdot (I-P_{Tg})e)=0$ but $g\cdot (I-P_{Tg})e\ge 0$, so the strict positivity of $T$ gives $$g-P_{Tg}g=g\cdot (I-P_{Tg})e=0$$
  making $P_{Tg}g=g$. Thus $g$ is in the band generated by $Tg$ and hence 
  the band generated by $g$ is  contained in the band generated by $Tg$, giving
  $P_{Tg}\ge P_g$.
   \end{proof}

The Dedekind complete Riesz space $E$ is said to be universally complete with respect to $T$ ($T$-universally complete), if, for each increasing net ${\parent{f_{\alpha}}}$ in $E_+$ with $\parent{Tf_{\alpha}}$ order bounded, we have that ${\parent{f_{\alpha}}}$ is order convergent in $E$.  In this case $E$ is an $R(T)$ module and $R(T)$ is an $f$-algebra with the multiplication discussed earlier, see \cite{krw}.

We recall Birkhoff's ergodic theorem for a $T$-universally complete Riesz space from \cite[Theorem 3.9]{ergodic paper}.

\begin{theorem}[Birkhoff's (Complete) Ergodic Theorem]
\label{complete riesz birkhoff theorem}
Let $(E,T,S,e)$ be a conditional expectation preserving system with $T$ strictly positive and $E$ $T$-universally complete then $E=\3E_S$ and hence $L_S=SL_S$. In addition, $TL_S=T$ and $L_S$ is a conditional expectation operator on $E$.
\end{theorem}

Applying the above theorem to the concept of conditional ergodicity as defined in Definition \ref{ergodicity def} we have that $\parent{E,T,S,e}$ is conditionally ergodic if and only $L_S f=Tf$ for all $f \in \3I_S$.
This leads to the following characterization of conditional ergodicity.

\begin{corollary} \label{tsm corollary}
Let $(E,T,S,e)$ be a conditional expectation preserving system with $T$ strictly positive and $E$ $T$-universally complete.
The conditional expectation preserving system $(E,T,S,e)$ is conditionally ergodic if and only if $T=L_S$.
\end{corollary}

Combining Theorems \ref{complete riesz birkhoff theorem}  and Corollary \ref{tsm corollary} we obtain the following corollary which fundamental to the proof of a conditional Riesz space version of the Kac formula,  Theorem \ref{kac-2}.

\begin{corollary} \label{tsm-inverse}
If  $(E,T,S,e)$ is a conditional expectation preserving system with $T$ strictly positive, $E$ $T$-universally complete and $S$ surjective then $(E,T,S^{-1},e)$ is a conditional expectation preserving system,  $E=\3E_S=\3E_{S^{-1}}$,  $R(T)\subset R(L_S)=\3I_S=\3I_{S^{-1}}=R(L_{S^{-1}})$ and $L_{S^{-1}}=L_S$.
\end{corollary}

\begin{proof}
As $E$ is $T$-universally complete,  by Theorem \ref{complete riesz birkhoff theorem},  $E=\3E_S$.

As $T$ is strictly positive,  $S$ is injective,  so by the surjectivity assumption,  $S$ is bijective.
Thus $S^{-1}$ exists and is a Riesz homomorphism.  Further $Se=e$ gives $e=S^{-1}e$ while from $TS=T$ we have $T=TSS^{-1}=TS^{-1}$ making $(E,T,S^{-1},e)$ a conditional expectation preserving system and again $E=\3E_{S^{-1}}$.

Finally $f \in \3I_S$ if and only if $Sf=f$ which is equivalent to (after applying $S^{-1}$) $f=S^{-1}f$.  Thus $\3I_S=\3I_{S^{-1}}$. For $f \in \3I_S$ and $L_Sf=f$ which  gives $f=L_Sf\in R(L_S)$ so $\3I_S\subset R(L_S)$. 
However, from  Theorem \ref{complete riesz birkhoff theorem} $SL_S=L_S$ giving that $R(L_S)\subset \3I_S$ so $R(L_S)= \3I_S$. Hence
$R(L_S)=\3I_S=\3I_{S^{-1}}=R(L_{S^{-1}})$.

For each component $p$ of $e$ with $p\in R(T)$ we have that $0\le L_Sp\le L_Se=e$ giving that $L_Sp\in E_e$. Thus $(p-L_Sp)^2\in E_e$ and by the averaging property of $L_S$, along with $p^2=p$, we have
$$(p-L_Sp)^2=(p^2-2p\cdot L_Sp+(L_Sp)^2=p-2p\cdot L_Sp+L_S(p\cdot L_Sp)$$
which after applying $T$ and using that $TLS=T$,  $Tp=p$ and the averaging property of $T$ gives
\begin{eqnarray*}
T(p-L_Sp)^2
&=&Tp-2T(p\cdot L_Sp)+TL_S(p\cdot L_Sp)\\
&=&p-T(p\cdot L_Sp)\\
&=&p-p\cdot TL_Sp\\
&=&p-p\cdot Tp=p-p^2=0.
\end{eqnarray*}
From strict positivity of $T$, $(p-L_S p)^2=0$ giving $p=L_Sp$,  Hence each component of $e$ which is in $R(T)$ is also in $R(L_S)$. But every element of $R(T)$ can be expressed as an order limit of a net of linear combinations of components of $e$ which are in $R(T)$ and $R(L_S)$ is a Dedekind complete Riesz subspace of $E$. Hence $R(T)\subset R(L_S)=\3I_S$. Now  from \cite{radon paper}, $L_S=L_{S^{-1}}$.
\end{proof}

%%%%%%%%%%%%%%%%%%%% Recurrence %%%%%%%%%%%%%%%%%%
\section{Poincar\'{e}'s recurrence theorem in Riesz Spaces} \label{sec-poin}
%-----------------------------------------------------------------------------------------------------------%

We begin by characterizing recurrence,  in Definition \ref{recurrence definition},  for the case of $S$ bijective.

\begin{lemma}\label{recurrence lemma}
 Let $\parent{E,T,S,e}$ be a conditional expectation preserving system with $S$ bijective, then a component $p$ of $q$ is recurrent with respect to $q$ a component of $e$ if and only if $$p\le\bigvee_{n\in\N} S^{-n}q.$$
\end{lemma}

\begin{proof}
 For $n\in\N$, let
 $$m_n=\bigvee \{g\in C_p(E)\,|\,S^ng\le q\}.$$
Here $m_n$ is the maximal component of $p$ with $S^nm_n\le q$ and hence the maximal component of $p$ to recur at $n$ iterates of $S$.
So in terms of $m_n$, $p$ is recurrent with respect to $q$ if and only if 
 $\displaystyle{\bigvee_{n\in\N} m_n \ge p.}$
 As $S$ is bijective, $m_n=p\wedge S^{-n}q$, making $p$ recurrent with respect to $q$ if and only if
 $\displaystyle{p\le \bigvee_{n\in\N} p\wedge S^{-n}q}$
 from which the result follows.
\end{proof}

 \begin{theorem}[Poincar\'{e}]\label{thm-poincare}
Let $(E,T,S,e)$ be a conditional expectation preserving system with $T$ strictly positive and $S$ surjective then each component $p$ of $q$ where $q$ is a component of $e$ is recurrent with respect to $q$. 
\end{theorem}

\begin{proof}
By Lemma \ref{injective} and the surjectivity of $S$,  $S$ is bijective.
From Lemma \ref{recurrence lemma},  so it suffices to prove that
$$p\le\bigvee_{n\in\N} S^{-n}q.$$
This however follows from
$\displaystyle{q\le \bigvee \limits_{n=1}^{\infty} S^{-n}q}$, which we prove below.

Here $\displaystyle{\bigvee \limits_{n=1}^{\infty} S^{-n}q}$ exists in $E$ as $E$ is Dedekind complete and $S^{-n}q\le S^{-n}e=e$ for all $n\in\N$.  Further $\displaystyle{\bigvee \limits_{n=1}^{\infty} S^{-n}q}$ is a component of $e$,  giving 
 that $$r:=q\wedge\left(e - \bigvee \limits_{n=1}^{\infty} S^{-n}q\right)=q\wedge \bigwedge^{j=-1}_{-\infty} (e - S^jq)\ge 0$$ is a component of $q$.  Thus $S^nr$ is a component of $e$ (as $S$ maps components of $e$ to components of $e$). 
Since $Se=e$ we have
$$0\le q\wedge S^nr=q\wedge S^nq\wedge \bigwedge\limits_{-\infty}^{j=n-1} \left(e-S^j q\right)\le q\wedge (e-q)=0$$
which,  after application of $S^{-n}$,  gives $r\wedge S^{-n}q=0$ for all $n \ge 1.$
However $0\le r\le q$ so $r\wedge S^{-n}r=0$ for all $n \ge 1.$
Applying $S^{n+k}$ to the above gives $S^{n+k}r\wedge S^kr=0$ for all $n\in\N$, $k\ge 0$,
making  $(S^n r)_{n \ge 0}$ a sequence of disjoint components of $e$.  Thus 
$$\sum \limits_{n=1}^{N} S^n r =\bigvee \limits_{n=1}^{N} S^n r \le e$$ 
for all $N \in \mathbb{N}.$
Now,  applying $T$ to the above equation gives 
$$NTr= \sum \limits_{n=1}^{N} T r  =\sum \limits_{n=1}^{N} TS^n r \le Te=e$$
for all $N\in\N$.
Since $E$ is Archimedean,  this gives $Tr=0$.  As $T$ is strictly positive  and $r\ge 0$ it follows that $r=0.$
\end{proof}

%%%%%%%%%%%%%%%%%%%%%%%%%%%%%%%%%%%%%%%%
\section{Kac's formula in Riesz spaces}\label{sec-kac}

\begin{lemma}
Let $(E,T,S,e)$ be a conditional expectation preserving system with $T$ strictly positive and $S$ surjective then 
\begin{equation}\label{N-1}
N_p(n)=p\bigwedge_{j=1}^n (e-S^{-j}p)
\end{equation}
and
\begin{equation}\label{N-2}
n_p=\sum_{k=1}^\infty k S^{-k}p\wedge N_p(k-1).
\end{equation}
\end{lemma}

\begin{proof}
From (\ref{np}) we have
\begin{eqnarray*}
N_p(n)&=&\bigvee\{q\in C_p(E)\,|\, S^{-j}p\wedge q =0 \forall j=1,\dots n\}\\
&=& \bigvee\{q\in C_p(E)\,|\,(e- S^{-j}p)\ge q  \forall j=1,\dots n\}\\
&=&\bigwedge_{j=1}^n p\wedge (e- S^{-j}p)
\end{eqnarray*}
and hence (\ref{N-1}).
Further
\begin{eqnarray*}
N_p(k-1)-N_p(k)&=&p\wedge (e-(e-S^{-k}p))\wedge\bigwedge_{j=1}^{k-1} (e-S^{-j}p)\\
&=&S^{-k}p\wedge N_p(k-1)
\end{eqnarray*}
from which (\ref{N-2}) follows.
\end{proof}

\begin{lemma}\label{kac-1}
 Let $\parent{E,T,S,e}$ be a conditional expectation preserving system where $T$ is strictly positive and $S$ is surjective.  
 Let $L_S$ be as defined in Theorem \ref{complete riesz birkhoff theorem}.
If $E$ is $T$-universally complete then for each component $p$ of $e$ we have that $n_p\in E$ and
$$L_Sn_p=L_S\left(\bigvee_{k=0}^\infty S^{-k}p\right).$$
\end{lemma}

\begin{proof}
 Let $r\in C_e(E)$.  For convenience we set
 $\displaystyle{\bigwedge_{j=1}^0 S^{-j}(e-r)=e}$.
 We begin by proving by induction on $n\in\N$ that
 \begin{eqnarray}\label{induction-2}
L_Sr=\sum_{k=1}^nL_S\left((e-r)\wedge\bigwedge_{j=1}^{k} S^{-j}r\right)+
 L_S\left(\bigwedge_{k=0}^{n} S^{-k}r \right).
 \end{eqnarray}
Since $L_SS=SL_S=L_S$ we have $L_S=L_SSS^{-1}=L_SS^{-1}$ and thus 
 $$L_Sr=L_SS^{-1}r=L_S((e-r)\wedge S^{-1}r)+L_S(r\wedge S^{-1}r),$$
from which it follows that (\ref{induction-2}) holds for $n=1$.
Further, 
 \begin{eqnarray*}
 L_S\left(\bigwedge_{k=0}^{n} S^{-k}r\right)&=&
 L_SS\left(\bigwedge_{k=0}^{n} S^{-k}r\right)\\
&=&L_S\left(\bigwedge_{k=1}^{n+1} S^{-k}r\right)\\
&=&L_S\left((e-r)\wedge\bigwedge_{k=1}^{n+1} S^{-k}r\right)
+L_S\left(r\wedge\bigwedge_{k=1}^{n+1} S^{-k}r\right)\\
&=&L_S\left((e-r)\wedge\bigwedge_{k=1}^{n+1} S^{-k}r\right)
+L_S\left(\bigwedge_{k=0}^{n+1} S^{-k}r\right)
 \end{eqnarray*}
 which if (\ref{induction-2}) holds for $n$ gives that (\ref{induction-2}) holds for $n+1$.
 Hence giving that (\ref{induction-2}) holds by induction for all $n\in\N$.
 
For $p\in C_e(E)$ we set $r=e-p$ in (\ref{induction-2}) to give
  \begin{eqnarray*}\label{induction-3}
L_S(e-p)
 &=&\sum_{k=1}^nL_S\left(p\wedge\bigwedge_{j=1}^{k} S^{-j}(e-p)\right)+
 L_S\left(\bigwedge_{k=0}^{n} S^{-k}(e-p)\right)\\
 &=&\sum_{k=1}^nL_S(N_p(k))+
 L_S\left(\bigwedge_{k=0}^{n} S^{-k}(e-p)\right).
   \end{eqnarray*}
   So by (\ref{np0}),
     \begin{eqnarray}\label{dc-1}
e=L_S(e-p)+L_Sp
 &=&\sum_{k=0}^nL_S(N_p(k))+
 L_S\left(\bigwedge_{k=0}^{n} S^{-k}(e-p)\right).
   \end{eqnarray}
   
   Here we note that $\displaystyle{\left(\sum_{k=0}^n N_p(k)\right)_{n\in\N_0}}$ is an increasing sequence in $E_u^+$ and by (\ref{dc-1}),
   $\displaystyle{T\left(\sum_{k=0}^nN_p(k)\right)\le e}$,
   so from the $T$-universal completeness of $E$, $\displaystyle{\left(\sum_{k=0}^n N_p(k)\right)_{n\in\N_0}}$ converges in order in $E$, i.e. $\displaystyle{n_p=\sum_{k=0}^\infty N_p(k)\in E}$. 
   Now taking the order limit as $n\to\infty$ in (\ref{dc-1}) gives
     \begin{eqnarray}\label{kac-eq-1}
e  &=&L_Sn_p+
 L_S\left(\bigwedge_{k=0}^\infty S^{-k}(e-p)\right).
   \end{eqnarray}
  However $Se=e$ so $e=S^{-k}e$ giving $S^{-k}(e-p)=e-S^{-k}p$ and thus
   \begin{equation}\label{kac-eq-2}
   \bigwedge_{k=0}^\infty S^{-k}(e-p)=\bigwedge_{k=0}^\infty (e-S^{-k}p)=e-\bigvee_{k=0}^\infty S^{-k}p.
   \end{equation}
    Combining (\ref{kac-eq-1}) and (\ref{kac-eq-2}), and using that $L_Se=e$ we have
 \begin{eqnarray*}
e  &=&L_Sn_p+e-L_S\left(\bigvee_{k=0}^\infty S^{-k}p\right)
    \end{eqnarray*}
          from which the lemma follows.
\end{proof}

As $TL_S=T$,  applying $T$ to 
$$L_Sn_p=L_S\left(\bigvee_{k=0}^\infty S^{-k}p\right).$$
in the above Lemma gives the following corollary.

\begin{corollary}
 Let $\parent{E,T,S,e}$ be a conditional expectation preserving system where $T$ is strictly positive and $S$ is surjective. 
If $E$ is $T$-universally complete then for each component $p$ of $e$ we have that $n_p\in E$ and
$$Tn_p=T\left(\bigvee_{k=0}^\infty S^{-k}p\right).$$
\end{corollary}

{\bf Example:} Take $E=\ell^1(4)$ with $e(x)=1$ for all $x\in\{1,2,3,4\}$, $
\tau(x)=\left\{\begin{array}{ll} 2,& x=1\\ 1,& x=2,\\ 4,& x=3,\\ 3,& x=4\end{array}\right.$ with
components of $e$, $p_i(x)=\delta_i(x)$ where $\delta$ is the Kronecker symbol. The surjective Riesz homomorphism $S$ on $E$ is given by $Sf(x)=f(\tau(x))$.
Set $Tp_i=\frac{1}{2}(p_1+p_2)$ if $i\in\{1,2\}$ and $Tp_i=\frac{1}{2}(p_3+p_4)$ if $i\in\{3,4\}$.
Take $q_1=p_1+p_2$ and $q_2=p_1+p_3$. 
\begin{itemize}
\item[(a)]A direct computation gives $n_{q_1}=q_1$ with
$Tq_1=Tn_{q_1}$ here $\bigvee_{k=0}^\infty S^{-k}q_1=q_1$ so indeed verifying Kac.
\item[(b)]
Now consider $n_{q_2}=2q_2$ and $Tq_2=\frac{1}{2}e$ so $Tn_{q_2}=2Tq_2$ but here
 $\bigvee_{k=0}^\infty S^{-k}q_2=e$ so Kac gives $Tn_{q_2}=Te=2Tq_2$ indeed verifying Kac.
 \end{itemize}
 
 This example shows that the conditional version of the Kac formula is not as simple as that for expectations (where only case (b) appears).

\begin{theorem}[Kac]\label{kac-2}
 Let $\parent{E,T,S,e}$ be a conditional expectation preserving system,  where $T$ is strictly positive,  $E$ is $T$-universally complete and $S$ is surjective. 
Let $L_S$ be as defined in Theorem \ref{complete riesz birkhoff theorem}.
For each $p$ a component of $e$ we have $$L_Sn_p=P_{L_Sp}e.$$
\end{theorem}

\begin{proof}
Let $$w=\bigvee_{k=0}^\infty S^{-k}p.$$
From the Lemma \ref{kac-1}, we have
$L_Sn_p=L_Sw$,
so it remains to shown that 
$L_Sw=P_{L_Sp}e.$
We recall that $R(L_S)=\3I_S=\3I_{S^{-1}}=R(L_{S^{-1}}).$
Here $w\le e$ is a component of $e$ and
$$S^{-1}w=\bigvee_{k=1}^\infty S^{-k}p\le \bigvee_{k=0}^\infty S^{-k}p=w.$$
Thus $w\le Sw$, so, by Note \ref{note-comp}, $Sw=w$ giving 
 $w\in \3I_{S}=R(L_S)$ and so $L_Sw=w$.

From the definition of $w$ we have $w\ge p$ so $w=L_Sw\ge L_Sp$ and thus $w=P_we\ge P_{L_Sp}e$.
But $P_{L_Sp}e\ge P_pe=p$, Lemma \ref{averaging}. Here $P_{L_Sp}e\in R(L_S)=\3I_S=\3I_{S^{-1}}$,  so $P_{L_Sp}e$ is $S^{-1}$ invariant. Hence
$$P_{L_Sp}e=S^{-k}P_{L_Sp}e\ge S^{-k}p$$
for each $k=0,1,2,\dots$. Taking suprema over $k=0,1,2,\dots$ gives
$$P_{L_Sp}e\ge w.$$
Hence
$P_{L_Sp}e= w.$
\end{proof}

Applying $T$ to the above result gives
$$Tn_p=TP_{L_Sp}e$$
while in the case of $\parent{E,T,S,e}$ being conditionally ergodic $L_S=T$ so the above theorem gives the following Corollary.

\begin{corollary}[Kac]\label{kac-3}
 Let $\parent{E,T,S,e}$ be a conditionally ergodic conditional expectation preserving system,  where $T$ is strictly positive,  $E$ is $T$-universally complete and $S$ is surjective. 
For each $p$ a component of $e$ we have $$Tn_p=P_{Tp}e.$$
\end{corollary}

%%%%%%%%%%%%%%%%%%%%%%%%%%%%%%%%%%%%%%%%%%%%%%
 
 \section{Application to Probabilistic Processes}\label{sec-appl}

 Consider a probability space $(\Omega,{\cal A},\mu)$ where $\mu$ is a complete measure (i.e. all subset of sets of measure zero are measurable).  Take $\Sigma$ a sub-$\sigma$-algebra of ${\cal A}$. Let $E$ denote the space of a.e. equivalence classes of measurable functions $f:\Omega\to \R$ for which the sequence $(\E[\min(|f(x)|, {\bf n})|\Sigma])_{n\in\N}$, is bounded above by an a.e. finite valued measurable function. 
Here ${\bf n}$ is the (equivalence class of the) function with value $n$ a.e.
  Then $E$ is the natural domain of the conditional expectation operator $T=\E[\cdot|\Sigma]$.  Here $L^1(\Omega,{\cal A},\mu)\subset E$
 and the extension of  $\E[\cdot|\Sigma]$ to $E$ is given by the a.e.  pointwise limits
 $$Tf=\lim_{n\to\infty}\E[\min(f^+,{\bf n})|\Sigma]-\lim_{n\to\infty}\E[\min(f^-,{\bf n})|\Sigma],\quad f\in E.$$
The space $E$ is a $T$-universally complete Riesz space  with  weak order unit ${\bf 1}$ and $T$ is a strictly positive Riesz space conditional expectation operator on $E$ having $T{\bf 1}={\bf 1}$.  If $\Sigma=\{ A\subset \Omega\,|\,\mu(A)=0 \mbox{ or } \mu(A)=1\}$ then $T$ is the expectation operator and $E=L^1(\Omega,{\cal A},\mu)$.

Let $\tau:\Omega\to\Omega$ be a map with $\tau^{-1}(A)\in{\cal A}$ (i.e. $\tau$ is $\mu$-measurable) and $\E[\chi_{\tau^{-1}(A)}|\Sigma]=\E[\chi_{A}|\Sigma]$, for all $A\in {\cal A}$. 
Further we require that for each $A\in {\cal A}$ there is $B_A\in {\cal A}$ so that $\mu(A\Delta \tau^{-1}(B_A))=0$.  
Now $Sf:=f\circ \tau$ is a surjective Riesz homomorphism on $E$ with $S{\bf 1}={\bf 1}$ and $TS=T$.  
The system $(E,T,S,e)$ is a conditional expectation preserving system, with $S$ surjective.  
Theorem \ref{complete riesz birkhoff theorem} gives that 
$$L_Sf=\lim_{n\to\infty}\frac{1}{n}\sum_{k=0}^{n-1}S^kf=\lim_{n\to\infty}\frac{1}{n}\sum_{k=0}^{n-1}f\circ \tau^k$$
converges a.e. pointwise to a conditional expectation operator on $E$ (which when restricted to 
$L^1(\Omega,{\cal A},\mu)$ is a classical conditional expectation operator).  The system $(E,T,S,e)$ is conditionally ergodic if and only if $L_S=T$ which is equivalent to $\tau^{-1}(A)=A$ with $A\in{\cal A}$,  if and only if $A\in \Sigma$.
Here $n_A=n_{\chi_A}$ in the a.e. sense on $A$, and 
applying Corollary \ref{kac-3} we obtain that
$$\E[n_A|\Sigma]=\chi_{\{x\in\Omega|\PP[A|\Sigma](x)>0\}}.$$
where $\PP[A|\Sigma]=\E[\chi_A|\Sigma]$ is the conditional probability of $A$ given $\Sigma$.
  
%----------------------------------------------------------%
%%%%%%%%%%%%%%%%%%%%%%% bibliography %%%%%%%%%%%%%%%%%%
\bibliographystyle{amsplain}

\begin{thebibliography}{99}



%--------------------------------------------------------------------------------%
\bibitem{riesz book}{
{\sc Aliprantis, C.D., Border, K.C.}, 
{\textit{Infinite Dimensional Analysis: A Hitchhiker's Guide}}, 
{Second Edition}, 
{Springer-Verlag, Berlin}, 
{1999}. 
}
	%--------------------------------------------------------------------------------%

\bibitem{AT}{\sc Azouzi, Y., Trabelsi, M.},
     {$L^p$-spaces with respect to conditional expectation on Riesz spaces},
  \textit{J. Math. Anal. Appl.}, {\textbf{447}}(2017), 798-816.

	%--------------------------------------------------------------------------------%

 \bibitem{dPG}{{\sc de Pagter, B., Grobler, J.J.},
         {Operators representable as multiplication-conditional expectation operators},
          {\em J. Operator Theory,} {\bf 48} (2002), 15-40.}

 \bibitem{HdP}{{\sc de Pagter, B., Huijsmans, C.B.},
         Averaging operators and positive contractive projections,
         {\em J. Math. Anal. Appl.}, {\bf 113}, (1986), 163-184.}

 \bibitem{dPDH}{{\sc de Pagter, B., Dodds, P.G., Huijsmans, C.B.},
         {Characterizations of conditional expectation-type operators},
          {\em Pacific J. Math.,} {\bf 141} (1990), 55-77.}

	%--------------------------------------------------------------------------------%
\bibitem{EFHN}{
{\sc Eisner, T., Farkas, B., Haase, M., Nagel, R.},
{\textit{Operator Theoretic Aspects of Ergodic Theory}},
{Springer-Verlag, Berlin}, 
{2016}.} 
	%-------------------------------------------------------------------------%
	%--------------------------------------------------------------------------------%
	
	
	\bibitem{grobler}{
{\sc Grobler, J.J.}, 
{Stopped processes and Doob's optional sampling theorem},
{\textit{J. Math. Anal. Appl.}},
{\textbf{497}}, 
{(2017), 124875}.}
%--------------------------------------------------------------------------------%


%--------------------------------------------------------------------------------%
\bibitem{hkw-1}{
{\sc Homann, J.M., Kuo, W-C., Watson, B.A.}, 
{A Koopman-von Neumann type theorem on the convergence of the Ces\`aro means in Riesz Spaces}, 
\textit{Proc. American Math. Soc., Series B}, {\bf 8}(2021), 75-85.}

 

%--------------------------------------------------------------------------------%
\bibitem{hkw-2}{
{\sc Homann, J.M., Kuo, W-C., Watson, B.A.}, 
{Ergodicity in Riesz Spaces}, 
\textit{Positivity and its
Applications,
Positivity X, 8-12 July 2019, Pretoria,
South Africa}, Springer Nature, 2021,
193-201.}



	%--------------------------------------------------------------------------------%
\bibitem{klw-1}{
{\sc Kuo, W-C., Labuschagne, C.C.A., Watson, B.A.}, 
{Discrete-Time stochastic processes on Riesz Spaces}, 
{\textit{Indag. Mathem., N.S.}}, 
{\textbf{15}}, 
{(2004), 435-451}
.}
	%--------------------------------------------------------------------------------%
\bibitem{expectation paper}{
{\sc Kuo, W-C., Labuschagne, C.C.A., Watson, B.A.}, 
{Conditional Expectations on Riesz Spaces}, 
{\textit{J. Math. Anal. Appl.}}, 
{\textbf{303}}, 
{(2005), 509-521}
.}
	%--------------------------------------------------------------------------------%
\bibitem{ergodic paper}{
{\sc Kuo, W-C., Labuschagne, C.C.A., Watson, B.A.}, 
{Ergodic Theory and the Strong Law of Large Numbers on Riesz Spaces}, 
{\textit{J. Math. Anal. Appl.}}, 
{\textbf{325}}, 
{(2007), 422-437}
.}
	
	\bibitem{krw}{
{\sc Kuo, W-C., Rogans, M.J., Watson, B.A.},  
\textit{Mixing inequalities in Riesz spaces.} 
{\textit{J. Math. Anal. Appl.}},
\textbf{456},
(2017),
992--1004.}
%--------------------------------------------------------------------------------%
    \bibitem{L-Z}{{\sc  Luxemburg, W.A.J., Zaanen,  A.C.},
         {\em Riesz Spaces I,}
         North Holland, 1971.}

	
	%--------------------------------------------------------------------------------%
\bibitem{petersen}{
{\sc Petersen, K.}, 
{\textit{Ergodic Theory}}, 
{Cambridge University Press, Cambridge}, 
{1983}%, 
%pg. 23-72
.}
	
	
 \bibitem{stoica}{{\sc Stoica, G.},
          {Martingales in vector lattices},
          {\em Bull. Math. de la Soc. Sci. Math. de Roumanie}, {\bf 34(82)} (1990), 357-362.}


	
 \bibitem{troitski}{{\sc Troitsky, V.},
         {Martingales in Banach lattices},
         {\em Positivity,} {\bf 9} (2005), 437-456.}

%--------------------------------------------------------------------------------%
\bibitem{radon paper}{
{\sc Watson, B.A.}, 
{An And\^{o}-Douglas Type Theorem in Riesz Spaces with a Conditional Expectation}, 
{\textit{Positivity}}, 
{\textbf{13}}, 
{(2009), 543-558}
.}
	%--------------------------------------------------------------------------------%
	
\end{thebibliography}



\end{document}